\pdfoutput=1
\documentclass[]{amsart}
\usepackage{amssymb}
\usepackage{mathrsfs}                     % Calligraphic font
\usepackage[all]{xy}% xymatrix for CD's
\usepackage{color}

\allowdisplaybreaks

\def\undersetbrace#1\to#2{\underbrace{#2}_{#1}}
\def\oversetbrace#1\to#2{\overbrace{#2}^{#1}}
\def\AMSunderset#1\to#2{\underset{#1}{#2}}
\def\AMSoverset#1\to#2{\overset{#1}{#2}}

\def\East#1#2{-\raisebox{0.1pt}{$\mkern-16mu\frac{\;\;#1\;}{\;\;#2\;}\mkern-16mu$}\to}

\swapnumbers
\newtheorem{proposition}[subsection]{Proposition}
\newtheorem*{proposition*}{Proposition}
\newtheorem{theorem}[subsection]{Theorem}
\newtheorem*{theorem*}{Theorem}
\newtheorem*{maintheorem*}{Main Theorem}
\newtheorem{lemma}[subsection]{Lemma}
\newtheorem*{lemma*}{Lemma}

\newtheorem*{corollary*}{Corollary}

\theoremstyle{definition}
\newtheorem*{remark*}{Remark}

\def\ign#1{}             %=ignore, invisible entry for the index only
\def\o{\circ}
\def\X{\mathfrak X}
\def\al{\alpha}
\def\be{\beta}
\def\ga{\gamma}

\def\ep{\varepsilon}

\def\et{\eta}

\def\la{\lambda}
\def\rh{\rho}

\def\ta{\tau}
\def\ph{\varphi}

\def\ps{\psi}
\def\om{\omega}

\def\Ph{\Phi}

\def\Om{\Omega}
\def\i{^{-1}}
\def\x{\times}
\def\p{\partial}
\let\on=\operatorname

\def\Diff{\on{Diff}}
\def\Dens{\on{Dens}}

\def\R{{\mathbb R}}

\newcommand{\sr}[1]%
{\ifmmode{}^\dagger\else${}^\dagger$\fi\ifvmode
\vbox to 0pt{\vss
 \hbox to 0pt{\hskip\hsize\hskip1em
 \vbox{\hsize3cm\raggedright\pretolerance10000
 \noindent #1\hfill}\hss}\vss}\else
 \vadjust{\vbox to0pt{\vss%
 \hbox to 0pt{\hskip\hsize\hskip1em%
 \vbox{\hsize3cm\raggedright\pretolerance10000%
 \noindent #1\hfill}\hss}\vss}}\fi%
}

\begin{document}%\topmatter
\title[]
{Moser's theorem on manifolds with corners
}
\author{Martins Bruveris, Peter W. Michor, Adam Parusi\'nski, Armin Rainer}

\address{
Martins Bruveris: 
Department of Mathematics, 
Brunel University London, Ux\-bridge, UB8 3PH, United Kingdom}
\email{martins.bruveris@brunel.ac.uk}

\address{
Peter W.\ Michor: Fakult\"at f\"ur Mathematik,
Universit\"at Wien, Os\-kar-Mor\-gen\-stern-Platz 1, A-1090 Wien, Austria.}
\email{peter.michor@univie.ac.at}

\address{
Adam Parusi\'nski: Univ. Nice Sophia Antipolis, CNRS, LJAD, UMR 7351, 06108 Nice, France}
\email{adam.parusinski@unice.fr}

\address{Armin Rainer: Fakult\"at f\"ur Mathematik, Universit\"at Wien,  Oskar-Morgenstern-Platz~1, A-1090 Wien, Austria}
\email{armin.rainer@univie.ac.at}

\date{{\today}} 

\thanks{Supported by the Austrian Science Fund (FWF), Grant P~26735-N25, and 
by a BRIEF Award from Brunel University London.}
%\subjclass[2010]{Primary 58B20, 58D15} 

\begin{abstract} 
Moser's theorem \cite{Moser65} 
states that the diffeomorphism group of a compact manifold acts transitively on the space of all smooth positive densities with fixed volume. Here we describe the extension of this result to manifolds with corners. In particular we obtain Moser's theorem on simplices. The proof is based on Banyaga's paper  \cite{Banyaga74}, where Moser's theorem is proven for manifolds with boundary. 
A cohomological interpretation of Banyaga's operator is given, which allows a proof of Lefschetz duality using differential forms.
\end{abstract}

\maketitle

\subsection{Introduction}
In \cite{Moser65} Moser proved that on a connected compact oriented manifold $M$ without boundary there exists for any two positive volume forms $\mu_0$ and $\mu_1$ with $\int_M\mu_0=\int_M\mu_1$ 
an orientation preserving diffeomorphism $\ph$ with $\ph^*\mu_1=\mu_0$.
In \cite{Banyaga74} Banyaga extended this to compact oriented manifolds with boundary and showed that the diffeomorphism can be chosen such that it restricts to the identity on the boundary. 
On a manifold with corners one cannot expect the diffeomorphism to be the identity on the boundary, since at a corner $x$ of index 2 or higher the derivative of such a diffeomorphism would have to be the identity: in this case $x$ lies in the boundary of at least two codimension 1 strata of $\p M$. So the derivative at $x$ restricted to two codimension 1 subspaces is the identity and thus it has to be the identity on the whole space; in particular the Jacobian determinant there equals 1.

Moser's theorem on manifolds with corners is needed for example in \cite{BMR16}. Even on simplices it does not seem to be known, but is highly desirable.
In fact, Banyaga's method \cite{Banyaga74} gives just the desired result. But this is not immediately obvious and it took us a long time to realize it. Therefore we think that it is worthwhile to write the proof with all details. Along the way we also prove Stokes' theorem on manifolds with corners.

For related results see \cite{DacorognaMoser90} for a version of Moser's theorem on bounded domains in $\R^m$ with low differentiability requirements furnishing diffeomorphisms with only low regularity using PDE techniques; this does not imply the result given here. See also the recent book \cite{CsatoDacorognaKneuss12} for results on $k$-forms instead of volume forms. 
A version on non-compact manifolds is in \cite{GreeneShiohama79} which also sketches a proof for non-compact manifolds with boundary which differs from Banyaga's proof.

\subsection{Manifolds with corners alias quadrantic (orthantic) manifolds}
For more information we refer to \cite{DouadyHerault73}, \cite{Michor80}, \cite{Melrose96}.
Let $Q=Q^m=\mathbb R^m_{\ge 0}$ be the positive orthant or quadrant. By Whitney's extension theorem or Seeley's theorem,
the restriction $C^{\infty}(\mathbb R^m)\to C^{\infty}(Q)$ is a surjective continuous linear mapping which admits a continuous linear section (extension mapping); so $C^{\infty}(Q)$ is a direct summand in $C^{\infty}(\mathbb R^m)$. A point $x\in Q$ is called a \emph{corner of codimension (or index)} $q>0$ if $x$ lies in the intersection of $q$ distinct coordinate hyperplanes. Let $\p^q Q$ denote the set of all corners of codimension $q$.

A manifold with corners (recently also called a quadrantic manifold) $M$ 
is a smooth manifold modeled on open subsets of $Q^m$.
We assume that it is connected and second countable; then it is paracompact and each open cover admits a subordinated smooth partition of unity. Any manifold with corners $M$ is a submanifold with corners of an open manifold $\tilde M$ of the same dimension, and each smooth function on $M$ extends to a smooth function on $\tilde M$. We do not assume that $M$ is oriented, but for Moser's theorem we will eventually assume that $M$ is compact. 
Let $\p^q M$ denote the set of all corners of codimension $q$. Then 
$\p^q M$ is a submanifold without boundary of codimension $q$ in $M$; 
it has finitely many 
connected components if $M$ is compact. 
We shall consider $\p M$ as stratified into the connected components of all $\p^q M$ for $q > 0$. 
Abusing notation we will call $\p^q M$ the boundary stratum of codimension $q$; this will lead to no confusion. Note that $\p M$ itself is not a manifold with corners.
We shall denote by $j_{\p^q M}:\p^q M\to M$ the embedding of the boundary stratum of codimension $q$ into $M$, and by $j_{\p M}:\p M\to M$ the whole complex of embeddings of all strata.

Each diffeomorphism of $M$ restricts to a diffeomorphism of $\p M$ and to a diffeomorphism of each stratum $\p^q M$. The Lie 
algebra of $\Diff(M)$ consists of all vector fields $X$ on $M$ such that $X|\p^q M$ is tangent to 
$\p^q M$. We shall denote this Lie algebra by $\X(M,\p M)$.

\subsection{Differential forms}
There are several differential complexes on a manifold with corners.  
If $M$ is not compact there are also the versions with compact support. 
\begin{itemize}
\item Differential forms that vanish near $\p M$. If $M$ is compact, this is the same as
the differential complex $\Om_c(M\setminus \p M)$ of differential forms with compact support 
in the open interior $M\setminus \p M$. 
\item $\Om(M,\p M) = \{\al\in \Om(M): j_{\p^q M}^*\al =0 \text{ for all } q\ge 1\}$, the complex of differential forms that pull back to 0 on each boundary stratum. 
\item $\Om(M)$, the complex of all differential forms. Its cohomology equals 
singular cohomology with real coefficients of $M$, since $\mathbb R\to \Om^0\to \Om^1\to \dots$
is a fine resolution of the constant sheaf on $M$; for that one needs existence of smooth partitions of unity and the Poincar\'e lemma which holds on manifolds with corners.
The Poincar\'e lemma can be proved as in \cite[9.10]{Mic2008} in each quadrant.
\end{itemize}
If $M$ is an oriented manifold with corners of dimension $m$ and if $\mu\in \Om^m(M)$ is a nowhere vanishing form of top degree, then $\X(M)\ni X\mapsto i_X\mu\in \Om^{m-1}(M)$ is a linear isomorphism. 
Moreover, $X\in \X(M,\p M)$ (tangent to the boundary) if and only if $i_X\mu\in\Om^{m-1}(M,\p M)$.

\subsection{Towards the long exact sequence of the pair $(M,\p M)$} 
Let us consider the short exact sequence of differential graded algebras
$$
0\to \Om(M,\p M) \to \Om(M) \to \Om(M)/\Om(M,\p M)\to 0\,.
$$
The complex $\Om(M)/\Om(M,\p M)$ is a subcomplex of the product  of $\Om(N)$ for all connected components $N$ of all 
$\p^q M$. The quotient consists of forms which extend continuously over boundaries to $\p M$ with its induced topology in such a way that one can extend them to smooth forms on $M$; this is contained in the space of `stratified forms' as used in \cite{Valette15}. There Stokes' formula is proved for stratified forms.

\begin{proposition}[Stokes' theorem] \label{Stokes}
For a connected oriented manifold $M$ with corners of dimension $\dim(M)=m$ and for any $\om\in\Om^{m-1}_c(M)$ we have
$$
\int_M d\om = \int_{\p^1M} j_{\p^1 M}^*\om\,.
$$
\end{proposition}

Note that $\p^1M$ may have several components. Some of these might be non-compact.

We shall deduce this result from Stokes' formula for a manifold with boundary by making precise the fact 
that $\p^{\ge 2} M$ has codimension 2 in $M$ and has codimension 1 with respect to $\p^1 M$. The proof also works for manifolds with more general boundary strata, like manifolds with cone-like singularities.
A lengthy full proof can be found in \cite{Conrad}.

\begin{proof}
We first choose a smooth decreasing function $f$ on $\mathbb R_{\ge 0}$ 
such that $f=1$ near 0 and $f(r)=0$ for $r\ge \ep$.
Then $\int_0^\infty f(r)dr <\ep$ and for $Q^m=\R^m_{\geq 0}$ with $m \geq 2$,
\begin{align*}
	\Big|\int_{Q^m} f'(|x|) \, dx\Big| &= C_m \Big|\int_0^\infty f'(r) r^{m-1} \, dr \Big| = 
	C_m \Big|\int_0^\infty f(r) (r^{m-1})' \, dr \Big|
	\\
	&= C_m  \int_0^\ep f(r) (r^{m-1})' \, dr \le C_m \ep^{m-1}\,, 
\end{align*}
where $C_m$ denotes the surface area of $S^{m-1} \cap Q^m$.
Given $\om\in\Om^{m-1}_c(M)$ we use the function $f$ on quadrant charts on $M$ to construct a function 
$g$ on $M$ that is 1 near $\p^{\ge2}M = \bigcup_{q\ge2}\p^q M$, has support close to $\p^{\ge 2}M$ and satisfies $\left| \int_M dg \wedge \om \right| < \ep$. Then $(1-g)\om$ is an $(m-1)$-form with compact support in the manifold 
with boundary $M\setminus \p^{\ge 2}M$, and Stokes' formula (cf.\ \cite[10.11]{Mic2008}) now says
$$
\int_{M\setminus \p^{\ge 2}M} d((1-g)\om) = \int_{\p^1M} j_{\p^1 M}^*((1-g)\om)\,.
$$
But $\p^{\ge 2}M$ is a null set in $M$ and the quantities
$$
\Big| \int_M d((1-g)\om) - \int_M d\om\Big|  \quad\text{ and }\quad
\Big| \int_{\p^1M} j_{\p^1 M}^*((1-g)\om) - \int_{\p^1M} j_{\p^1 M}^*\om\Big| 
$$
are small if $\ep$ is small enough.
\end{proof}

\begin{lemma}\label{lemBanyaga} 
Let $M$ be an oriented connected manifold with corners of dimension $\dim(M) = m$.
For each form $\om \in \Om^m_c(M\setminus\p M)$ with $\int\om=1$ there exists a continuous linear operator 
$$
I^\om : \Om^m_c(M) \to \Om^{m-1}_c(M,\p M) \quad\text{ such that: }
$$
\begin{itemize}
\item $d \,I^\om(\al) = \al - \om \int\al$  for all $\al\in \Om^m_c(M)$.
\item If $\al$ vanishes on $\p^{\ge 2}M$, i.e., $\al_x=0$ for all $x\in \p^{\ge 2}M$, then $I^\om(\al)$ vanishes on $\p M$.
\end{itemize}
\end{lemma}

For a compact oriented manifold with boundary, this is due to Banyaga \cite{Banyaga74}. We call $I^\om$ the \emph{Banyaga operator}.

\begin{proof}
	We first construct $I^\om_m$ for the case when $M$ is a 
partial quadrant $Q^m_p := \R^p_{\geq 0} \x \R^{m-p} = \{x\in \mathbb R^m\,:\, x^1\ge 0,\dots,x^p\ge 0\}$.

We construct $I^\om_m$ by induction on the dimension $m$ and start with $I^\om_0=0$. We shall use a 
smooth function $g$ with compact support in $\mathbb R_{>0}$ and $\int g(u)du =1$. 

For $Q=Q^1_1 = \mathbb R_{\ge 0}$ let $\om=g(u)du$.  
Then for $\al=a(u)du\in \Om^1_c(Q)$ we put
\begin{align*}
 I_1^{\om}(\al)(u) :&= \int_0^u\Big( a(t) - g(t)\int_Q a(v)dv\Big)dt\, \quad\text{ so that }
\\
d I^\om_1(\al) &= \al - \om \int\al \quad \text{ and  }\quad  I^\om_1(\al)(0) = 0\,,
\end{align*}
and thus $I^\om_1(\al)\in \Om^0(\mathbb R_{\ge 0},\{0\})$.
For $Q=Q^1_0=\mathbb R$ we just integrate from $-\infty$ to $u$. Note that $I^\om_1(\al)(u)$ vanishes for large $u$, 
so it has compact support. Thus $I^\om_1$ has all desired properties.

For general $Q=Q^m_p = \mathbb R_{\ge 0}^p\x \mathbb R^{m-p}$ we shall use:
\begin{align*}
\om &= g(u^1)du^1\wedge \dots\wedge  g(u^m) du^m  =: \om'\wedge g(u^m)du^m\,,
\\
d : & ~\Om^{m-1}(Q)\to\Om^{m}(Q),\quad d = \sum_{i=1}^{m-1}du^i\wedge \p_{u^i} + du^m\wedge \p_{u^m} =: d' + du^m\wedge \p_{u^m}\,.
\end{align*}
Any form $\al\in\Om^m_c(Q^m_p)$ can be written as $\al=\al_1(u^m)\wedge du^m$ for a smooth curve 
$$
\al_1:
\begin{cases}
\mathbb R\to \Om^{m-1}_c(Q^{m-1}_p) &\quad\text{ if }\quad 0\le p\le m-1\,,
\\
\mathbb R_{\ge 0}\to \Om^{m-1}_c(Q^{m-1}_{p-1})&\quad\text{ if }\quad p=m\,.
\end{cases}$$ 
Following an idea of de~Rham \cite{deRham61} used by \cite{Banyaga74}, we define the auxiliary operator 
\begin{multline*}
\tilde I^\om_m(\al) = I^{\om'}_{m-1}(\al_1(u^m))\wedge du^m +
\\
+ (-1)^{m-1}  \om' \cdot
\begin{cases}
\int_{-\infty}^{u^m} \Big( \int_{Q^{m-1}_p}\al_1(t)- g(t)\int_{Q^m_p}\al\Big)\,dt &\quad\text{ if }\quad 0\le p\le m-1\,,
\\
\int_0^{u^m} \Big( \int_{Q^{m-1}_{p-1}}\al_1(t)- g(t)\int_{Q^m_p}\al\Big)\,dt &\quad\text{ if }\quad p=m\,,
\end{cases}
\end{multline*}
which is in $\Om^{m-1}(Q^m_p,\p Q^m_p)$: The first summand by induction on $m$ and because it contains $du^m$.
The second summand since the integral starts at $0$ in the relevant case.  $\tilde I^\om_m(\al)$ has compact support: The first summand by induction, and the second summand since $\om'$ has compact support in the first $m-1$ variables, and since the integral vanishes for large $u^m$.

The exterior derivative of $\tilde I^\om_m(\al)$: For the first summand we get  
\begin{align*}
d\big(I^{\om'}_{m-1}(\al_1(u^m))\wedge du^m\big) &= d' I^{\om'}_{m-1}(\al_1(u^m)) \wedge du^m =
\\&
= \Big(\al_1(u^m) -\om' \int_{Q^{m-1}} \al_1(u^m)\Big) \wedge du^m
\quad\text{ by induction}
\\&
=\al - (\om'\wedge du^m)\int_{Q^{m-1}}\al_1(u^m)\,.
\end{align*}
The exterior derivative of the second summand is
\begin{align*}
&(-1)^{m-1}du^m\wedge \om'\Big(\int_{Q^{m-1}}\al_1(u^m)-g(u^m)\int_{Q^m}\al\Big)
\\&
=(\om'\wedge du^m) \int_{Q^{m-1}}\al_1(u^m) - (\om'\wedge g(u^m)du^m) \int_{Q^m}\al 
\end{align*}
which proves $d\tilde I^\om_m(\al) = \al - \om\int_{Q^m}\al$.

Let $h : \R \to \R$ be a smooth function equal to $1$ for $t\le 1$ and to $0$ for $t\ge 2$.
Then we define 
$$
I^\om_m(\al) = 
\begin{cases} \tilde I^\om(\al) &\quad\text{ if }\quad 0\le p\le m-1\,,
\\
\tilde I^\om_m(\ga) + \rh^*\tilde I^\om_m(\rh^* \be) &\quad\text{ if }\quad p=m\,,
\end{cases}
$$
where $\beta = h (u^m)  \al_1(0)\wedge du^m$ and 
$\gamma = \alpha - \beta$, and where $\rh:Q^m_m\to Q^m_m$ is the permutation of the last two variables $u^m$ and $u^{m-1}$.  Both $\be$ and $\ga$ have compact support. 
We also have $d \,I^\om_m(\al) = \ga  - \om \int\ga + (\rh^*)^2\be - \rh^*\om\int\rh^*\be = \al -\om\int\al$. 

If  $\al$ vanishes on $\p^{\ge 2}Q^m$, then $\al_1(u^m)$ vanishes on $\p^{\ge 2}Q^{m-1}$ in both cases. 
If $p<m$ then $\tilde I^\om_m(\al)$ vanishes on $\p Q^m$ by induction and so does $I^\om_m(\al)$. 
If $p=m$ then 
$I^{\om'}_{m-1}(\al_1(u^m))$ need not vanish on $\{u^m=0\}\cap Q^m \subseteq \p Q^m_m$, and thus 
$\tilde I^{\om}_{m}(\al)$ need not vanish everywhere on $\p Q^m_m$. 
Clearly, if $\al_1(u^m)$ vanishes on $\{u^m=0\}\cap Q^m$, then 
so does $I^{\om'}_{m-1}(\al_1(u^m))$. 
If $\al$ vanishes on $\p^{\ge 2}Q^m_m$, then so do $\be$ and $\ga$. 
Moreover, $\gamma$ vanishes on $\{u^m=0\} \cap Q^m_m$ and $\beta$ vanishes on each $\{u^i=0\} \cap Q^{m}_m$ 
for $i<m$, thus $\rh^*\be$ vanishes on $\{u^m=0\} \cap Q^m_m$.
The second summand in the definition of $\tilde I^\om_m(\al)$ causes no problems, since the integral starts at 0 in the relevant case and $\rh^*\om=-\om$.
Consequently, $I^\om_m(\al)$ vanishes on $\p Q^m_m$ if $\al$ vanishes on $\p^{\ge 2} Q^m_m$. 
This finishes the construction of $I^\om_m$ by induction. 

In order to change to another $m$-form $\tilde\om\in \Om^m_c(Q^m_p\setminus\p Q^m_p)$
with $\int_{Q^m_p}\tilde\om =1$ we put
$$
I^{\tilde\om}_m(\al) = I^\om_m(\al) - I^\om(\tilde\om)\int_{Q^m_p}\al\,. 
$$

Now we extend the operators $I^\om_m$ to the oriented manifold with corners $M$. 
We construct an oriented atlas similarly to \cite[lemme 1]{Banyaga74} with the property that all charts contain a common chart $U_0$. 
Choose $x_0\in M\setminus\p M$ and a closed neighborhood $V_0$ of $x_0$ in $M\setminus \p M$ which is diffeomorphic to a closed ball in $\mathbb R^m$.
For each $y\in M\setminus V_0$ choose an oriented open chart $\ph_y:U_y \to Q^m_{p_y}$ centered at $y$ onto some partial quadrant, $x_y\in U_y\setminus \p U_y$ and a smooth embedded curve $c_y$ in 
$M\setminus \p M$ from $x_y$ to $x_0$. Then choose a vector field $X_y$ with $X_y(c_y(t))= c_y'(t)$ for each $t$ that vanishes at $y$ and on $\p M$. 
The flow $\on{Fl}^{X_y}_t$ moves $x_y$ along $c_y$ to $x_0$ and keeps $y$ and $\p M$ fixed.
$\on{Fl}^{X_y}_1$ also maps an open neighborhood of $x_y$ in $U_y$ to an open neighborhood of $x_0$, which we may extend to an open neighborhood of $V_0$ via a   
diffeomorphism $\ps_y$ of $M$ that is the identity near $y$ and near $\p M$.  
Now consider the charts
$$
\ps_y(\on{Fl}^{X_y}_1(U_y))\East{\ps_y\i}{} \on{Fl}^{X_y}_1(U_y) \East{\on{Fl}^{X_y}_{-1}}{} U_y \East{\ph_y}{}  Q^m_{p_y}
$$
and call the resulting atlas again $(U_y,\ph_y)$.
We choose a smooth partition of unity $\la_y$  with a locally finite family of supports subordinated to this atlas (most of the $\la_y$ are 0).
Finally we choose the chart $(U_0,\ph_0)$ inside $\bigcap_y U_y$ which is possible since the intersection contains the neighborhood $V_0$.

Choose $\om\in \Om^m_c(U_0)$ with $\int_M\om =1$ and let 
\begin{align*}
I^\om(\al) :={}& \sum_{y} \ph_y^* I^{(\ph_y\i)^*\om}_m\big((\ph_y\i)^*(\la_y.\al)\big) \in \Om^{m-1}_c(M,\p M)
\quad\text{ with }
\\
dI^\om(\al) ={}& \sum_{y} \ph_y^* d I^{(\ph_y\i)^*\om}_m\big((\ph_y\i)^*(\la_y .\al)\big)
\\
={}& \sum_{y} \ph_y^* \Big((\ph_y\i)^*(\la_y.\al) -  (\ph_y\i)^*\om\cdot \int_{Q_y}(\ph_y\i)^*(\la_y.\al)\Big)
= \al-\om\int_M\al\,.
\end{align*}
The sum is finite since $\al$ has compact support.
The change to an arbitrary form $\tilde \om\in \Om^m_c(M\setminus \p M)$
with $\int\tilde\om=1$ is as above.
 \end{proof}

\begin{theorem}[Moser's theorem for manifolds with corners]
Let $M$ be a compact connected smooth manifold with corners, possibly non-orientable. 
Let $\mu_0, \mu_1 \in \Dens_+(M)$ be smooth positive densities 
 with $\int_M\mu_0 = \int_M\mu_1$.
Then there exists a diffeomorphism $\ph:M\to M$ such that 
$\mu_1= \ph^*\mu_0$.
Moreover, $\ph$ can be chosen to be the identity on $\p M$ 
if and only if $\mu_0=\mu_1$ on $\p^{\ge 2}M$.
\end{theorem}

\begin{proof} 
We first prove the theorem for oriented $M$. 
In this case $\on{Dens}_+(M)$ equals the space $\Om^m_+(M)$ of positive $m$-forms for $m=\dim(M)$.
Put $\mu_t:=\mu_0+t(\mu_1-\mu_0)$ for $t\in[0,1]$; then each $\mu_t$ 
is a volume form on $M$ since these form a convex set. 
We look for a curve of diffeomorphisms, $t\mapsto \ph_t$, with 
$\ph_t^*\mu_t=\mu_0$; this curve has to satisfy $\frac{\partial}{\partial t}(\ph_t^*\mu_t)=0$.
Since $\int_{M}(\mu_1-\mu_0)=0$, we have 
$[\mu_1-\mu_0]=0\in H^m(M)$. 
Fix $\om \in \Om^m_c(M \setminus \p M)$ with $\int \om = 1$. Using lemma \ref{lemBanyaga} we have 
\begin{align*}
\ps :={}& I^\om(\mu_1-\mu_0) \in \Om^{m-1}(M,\p M)\quad\text{ with }
\\
d\ps ={}& d I^\om(\mu_1-\mu_0) = \mu_1-\mu_0 - \om\int_M( \mu_1-\mu_0) = \mu_1-\mu_0\,.
\end{align*}
Put $\et_t:=(\frac{\partial}{\partial t}\ph_t)\o \ph_t\i$; then by well known formulas
(see \cite[31.11]{Mic2008}, e.g.) we have:
\begin{align*}
0 &\AMSoverset\text{wish}\to= \tfrac{\partial}{\partial t}(\ph_t^*\mu_t) 
     = \ph_t^*\mathcal{L}_{\et_t}\mu_t + \ph_t^*\tfrac{\partial}{\partial t}\mu_t 
     = \ph_t^*(\mathcal{L}_{\et_t}\mu_t + \mu_1 - \mu_0)\,, \\
0 &\AMSoverset\text{wish}\to= \mathcal{L}_{\et_t}\mu_t + \mu_1 - \mu_0 
     = d i_{\et_t}\mu_t + i_{\et_t}d\mu_t + d\ps 
     = d i_{\et_t}\mu_t + d\ps\,. 
\end{align*}
We can choose $\et_t$ uniquely by requiring that $i_{\et_t}\mu_t=-\ps$, since $\mu_t$ 
is non-degenerate for all $t$. The time dependent vector field $\et_t$ is tangent to each boundary stratum $\p^q M$, 
since $\ps\in \Om(M,\p M)$.
Then the evolution operator 
$\ph_t=\Ph^\et_{t,0}$ 
exists for $t\in [0,1]$ since $M$ is compact, by  \cite[3.30]{Mic2008}. 
Moreover, $\ph_t:\p M\to\p M$. 
Thus $\ph_t$ restricts to a diffeomorphism of $M$ for each $t$.
On $M$ we have, using  \cite[31.11.2]{Mic2008}, 
\begin{equation*}
\tfrac{\partial}{\partial t}(\ph_t^*\mu_t) 
     = \ph_t^*(\mathcal{L}_{\et_t}\mu_t + d\ps) 
     = \ph_t^*(d i_{\et_t}\mu_t + d\ps) = 0\,,
\end{equation*}
so $\ph_t^*\mu_t= \text{ constant }= 
\mu_0$.
If $\mu_0=\mu_1$ on $\p^{\ge 2}M$, then $\ps = I^\om(\mu_1-\mu_0)$ vanishes on $\p M$ by lemma  \ref{lemBanyaga} and hence so does $\et_t$, thus $\ph_t$ is the identity there.

If $M$ is not orientable, we let
$p:\on{or}(M)\to M$ be the 2-sheeted orientable double cover of $M$:
It is the $\mathbb Z_2$-principal bundle with cocycle of transition functions 
$ \on{sign} \det d(u_\be\o u_\al\i)(u_\al(x))$
where $(U_\al,u_\al)$ is a smooth atlas for $M$. Each connected (thus orientable) 
chart of $M$ appears twice as chart of $\on{or}(M)$, once with each orientation.
Thus $\on{or}(M)$ is again a smooth manifold with corners. 
Let $\ta:\on{or}(M)\to \on{or}(M)$ be the orientation reversing deck-transformation; 
see \cite[13.1]{Mic2008}. Pullback $p^*:\Om(M)\to \Om(\on{or}(M))$ is an isomorphism onto the eigenspace 
$\Om(\on{or}(M))^{\ta^*=1}$ of $\ta^*$ with eigenvalue $1$.
The space $\on{Dens}_+(M)$ of positive smooth densities on $M$ is via $p^*$ isomorphic to the space 
of positive $m$-forms in the eigenspace $\Om^m(\on{or}(M))^{\ta^*=-1}$ of $\ta^*$ with eigenvalue $-1$; these are the `formes impaires' of de~Rham.  Note the abuse of notation here: $p^*$ of a density differs (by local signs) from $p^*$ of a form. 
See \cite[13.1 and 13.3]{Mic2008} for more details.

We consider the pullback densities $p^*\mu_t$ as positive $m$-forms denoted by $\nu_t$ on $\on{or}(M)$ 
which satisfy $\ta^*\nu_t= -\nu_t$, and for $\om\in\Om^m_c(\on{or}(M)\setminus \p\on{or}(M))$ with $\int_{\on{or}(M)}\om =1$ we choose 
\begin{align*}
\tilde\ps &= I^{\om}(\nu_1-\nu_0)\in \Om^{m-1}(\on{or}(M),\p \on{or}(M))
\quad\text{ which satisfies }
\\
d\tilde\ps &=\nu_1-\nu_0 - \om\cdot \int_{\on{or}(M)}(\nu_1-\nu_0) = \nu_1-\nu_0\,.
\end{align*} 
Let $\ps=\frac12\tilde\ps -\frac12\ta^*\tilde\ps$, then  again $d\ps = \nu_1-\nu_0$ and now also $\ta^*\ps = -\ps$. 
A time-dependent vector field $\et_t$ is uniquely given by $i_{\et_t}\nu_t = -\ps$.
\begin{equation*}
i_{\ta^*\et_t}\nu_t = - i_{\ta^*\et_t}\ta^*\nu_t = -\ta^*(i_{\et_t}\nu_t) = \ta^*\ps = -\ps 
\quad\implies\quad \ta^*\et_t = \et_t\,.
\end{equation*}
In particular, the time dependent vector field $\et_t$ is tangent to each boundary stratum $\p^q\on{or}(M)$, and it
projects to a time dependent vector field on $M$ whose evolution gives the curve of diffeomorphisms with all required properties. 
\end{proof}

\subsection{Cohomological interpretation of the Banyaga operator} 
For a connected oriented manifold with corners $M$ of dimension $m$ (we assume that $\p M$ is not empty) we consider the following diagram where only the dashed arrow $I^\om$ does not fit in commutingly. 
Here $\om\in\Om^m_c(M\setminus \p M)$ is a fixed form with $\int\om = 1$.
All instances of $\mathbb R$ in the diagram are connected by identities which fit commutingly into the diagram.
Each line is the definition of  the corresponding top de~Rham cohomology space. 
The integral in the first line induces an isomorphism in cohomology since 
$M\setminus \p M$ is a connected oriented open manifold.
The bottom triangle commutes by Stokes' theorem \ref{Stokes}.
$$
\xymatrix{
\Om^{m-1}_c(M\setminus \p M) \ar[r]^{d} \ar@{^{(}->}[d] & 
\Om^m_c(M\setminus \p M) \ar@/^1.5pc/[rr]^{\int_{M\setminus \p M}} \ar@{->>}[r]  \ar@{^{(}->}[d] &
H^m_c(M\setminus \p M) \ar@{=}[r]  \ar[d]& \mathbb R
\\
\Om^{m-1}_c(M, \p M) \ar[r]^{d} \ar@{^{(}->}[d] & 
\Om^m_c(M,\p M) \ar@{->>}[r] \ar@{=}[d] \ar@/_1.5pc/[rr]_{\qquad\qquad\qquad\qquad\int_M} &
H^m_c(M,\p M) \ar@{=}[r] \ar[d]&  \mathbb R
\\
\Om^{m-1}_c(M) \ar[r]^{d} \ar[dr]_{\int_{\p^1 M}\o j_{\p^1 M}^*} & 
\Om^m_c(M)  \ar@{->>}[r] \ar[d]^{\int_M}  \ar@{-->}[ul]_{I^\om} &
H^m_c(M) \ar@{=}[r]& 0
\\
& \mathbb R  
}
$$

\noindent{\bf Claim.} {\it $\int_M:\Om^m_c(M,\p M)=\Om^m_c(M) \to \mathbb R$ induces $H^m_c(M,\p M)=\mathbb R$.} \\
Namely, given $\al, \be \in \Om^m_c(M, \p M)$ with $\int \al = \int \be$, we have $\al - dI^\om(\al) = \om \int_M \al$ and similarly for $\be$. This implies $\al - \be = d\left(I^\om(\al) - I^\om(\be)\right)$ and hence $[\al] = [\be]$ in $H^m_c(M,\p M)$.
 
\noindent{\bf Claim.} {\it If $M$ has non-empty boundary then 
$H^m_c(M)=0$.} \\ 
For any form $\al\in\Om^m_c(M)$ we have $\al - dI^\om(\al)= \om\int_M\al$, so $\al$ equals a multiple of $\om$ modulo an exact form with compact support. 
Now choose $\be\in \Om^{m-1}_c(M)$ with $\int_M d\be\ne 0$; for example with $\int_{\p^1 M}j_{\p^1M}^*\be \ne 0$. 
Then $d\be - dI^\om(d\be) = \om \int_M d\be$ shows that any multiple of $\om$ is exact. Thus $H^m_c(M)=0$.

\begin{theorem}[Poincar\'e--Lefschetz duality]
For an oriented connected manifold with corners of dimension $m$ the cohomological integral 
$\int_* : H^m_c(M,\p M)\to \mathbb R$ induces a non-degenerate bilinear form
\begin{align*}
&P^k_M: H^k(M) \x H^{m-k}_c(M,\p M) \to \mathbb R\qquad\text{ given by }
\\&
P^k_M ([\al],[\be]) = \int_*[\al]\wedge [\be] = \int_M\al\wedge \be\,.
\end{align*}
\end{theorem}

This is in fact the special case for real coefficients of Lefschetz' duality \cite{Lefschetz26}.
In \cite{Valette14} Lefschetz duality is proven 
for piecewise linear stratified $\p$-pseudomanifolds in terms of intersection homology.
Here we can give a proof based completely on differential forms.

\begin{proof}
Note first that $\Om_c(M,\p M)$ is a graded ideal in $\Om(M)$, thus the integral makes sense. 
The proof follows now, for example, \cite[12.14 -- 12.16]{Mic2008} with some obvious changes.
\end{proof}

%\bibliographystyle{abbrv}
%\bibliography{Fisher-Rao}%preprints,../../ref/articles}

\end{document}